\documentclass[preprint,12pt]{elsarticle}

\usepackage{amsmath, amsthm, amssymb, bm, graphicx, mathrsfs,xcolor}
\numberwithin{equation}{section}
\newtheorem{Th}{Theorem}[section]
\newtheorem{Le}{Lemma}[section]

\journal{Inverse Problems and Imaging}

\begin{document}

\begin{frontmatter}



\title{On uniqueness of elastic scattering from a cavity}


\author[inst1]{Tianjiao Wang}

\affiliation[inst1]{organization={School of Mathematical Sciences},
            addressline={Zhejiang University}, 
            city={Hangzhou},
            postcode={310058}, 
            country={P. R. China}}

\author[inst2]{Yiwen Lin}
\affiliation[inst2]{organization={School of Mathematical Sciences and Institute of Natural
		Sciences},
	addressline={Shanghai Jiao Tong University}, 
	city={Shanghai},
	postcode={200240}, 
	country={P. R. China}}

\author[inst1,inst3]{Xiang Xu}

\fntext[inst3]{Corresponding author, xxu@zju.edu.cn. This work was supported in part by National Natural Science Foundation of China (11621101, 12071430, 12201404),  Key Laboratory of Collaborative Sensing and Autonomous Unmanned Systems of Zhejiang Province and Postdoctoral Science Foundation
	of China (2021TQ0203).}
\begin{abstract}
	The paper considers direct and inverse elastic scattering from a cavity in homogeneous medium with Dirichlet and Neumann boundary conditions. For direct scattering, existence and uniqueness are derived by variation approach.  For inverse scattering, Fr$\acute{\rm e}$chet derivatives of the solution operators are  investigated, which give local stability for Dirichlet case.
\end{abstract}

\begin{keyword}
	elastic cavity \sep variation formulation \sep Navier equations \sep existence \sep uniqueness
	\MSC  35P25 \sep  35A15 \sep 35R30
\end{keyword}

\end{frontmatter}


\section{Introduction}
\label{sec:sample1}
In this paper, we consider a time-harmonic plane elastic wave incident on a cavity which can be regarded as a local perturbation below the plane with wide-ranging applications as engine inlet ducts, cracks, gaps and so on.

For electromagnetic incident waves, there are considerable research results in literature. For instance,  Ammari, Bao and Wood proved  uniqueness and existence of electromagnetic cavity problems respectively to TE and TM polarization by variational approaches in \cite{r5}, and by integral equation methods in \cite{r16}. They also extend similar results for Maxwell's equations in \cite{r17} by variational approaches. Recently, Bao et al established a stability result explicitly dependent on wave number for TE polarization based on Fourier analysis in energy space in \cite{r21} and TM polarization in \cite{r22}. Numerical methods for large cavities can be seen in \cite{r18,r19,r20}.

Compared to electromagnetic scattering, elastic scattering also has wide applications in seismology and geophysics such as \cite{r12,r13,r14} and attracts more attention recently. However, elastic scattering problems have not been studied intensively due to their inherent difficulties arising from the governing equations and boundary conditions. For periodic surface, Arens studied quasi-periodic Green tensor and elastic potentials in \cite{r3} and used Rayleigh expansions and Rellich identities to prove uniqueness and integral equation methods to prove existence in \cite{r2}. Elschner and Hu applied variational approaches to get similar results in \cite{r6}. For general unbounded rough surfaces, Arens studied Green tensor, elastic potentials and upward radiation condition (UPRC), and proved uniqueness by integral equation methods and existence by integral equation methods, see \cite{r1,r4,r15}. Elschner and Hu \cite{r7} gave an equivalent form of UPRC and proved uniqueness and existence by variational approaches. They \cite{r11} also considered solvability in weighted Sobolev spaces and proved existence and uniqueness of elastic scattering from unbounded rough surface with an incident wave.
Moreover, Hu, Yuan and Zhao \cite{r8} combined variation formula and integral equation to prove uniqueness and existence of two-dimensions local perturbed surface scattering problem with Lipschitz graph. For cavities, Hu, Li and Zhao considered elastic scattering in three-dimensions and proved the uniqueness in \cite{r9} for Dirichlet boundary conditions.

Inverse cavity problem is to reconstruct the shape of a cavity by measured far-field data on the artificial boundary. In general, to reconstruct a shape,  Fr$\acute{\rm e}$chet derivative is commonly used. For instance,  Bao, Gao and Li considered TE and TM polarization in \cite{r23} and  proved uniqueness for lossy medium and studied the domain derivative and local stability for inverse electromagnetic cavity. A similar results for Maxwell's equations was extended in \cite{li2012inverse}. Bao and Lai \cite{bao2014radar} proposed an optimization scheme to design a cavity for radar cross section reduction. Recently, Hu, Yuan, Zhao considered domain derivative for the inverse elastic scattering by a locally perturbed rough surface in \cite{r8}. To the authors' best knowledge, research on inverse elastic scattering from cavities is still at its infant stage.  

This paper intends to extend electromagnetic scattering from a cavity to elastic scattering. By employing transparent boundary conditions given by Dirichlet to Neumann (DtN) and Neumann to Dirichlet (NtD) operators, the boundary value problems can be reduced into a bounded domain so that the Fredholm alternative can be applied. Then, the variational approach is used to prove existence and uniqueness of the elastic scattering problem in two-dimensions with Dirichlet condition or Neumann condition. For Dirichlet problem, the existence in two-dimensions is actually trivial compared to three dimensions in \cite{r9}. However, we can obtain the uniqueness for two-dimensions while it is still open for three dimensions. Neumann problem is similar to Dirichlet problem with additional difficulties for the NtD operator. On inverse elastic cavity problem, this paper applies the method of changing variable to give the Fr$\acute{\rm e}$chet derivative for Dirichlet problem and Neumann problem, respectively. The difficulty for inverse Neumann problem lies in  the boundary value problem reduced by it in Section 2 has inhomogeneous Neumann boundary condition by directly applying NtD operator. Hence, in order to obtain the TBC deduced by the {DtN} operator, we use an artificial boundary to rewrite the boundary value problem for Neumann case. 

The paper is outlined as follows. In Section 2, the direct problems are reduced into variation problems in a bounded domain by DtN operator and NtD operator. In Section 3 and Section 4, existence and uniqueness for Dirichlet problem and Neumann problem are examined respectively. Moreover, the solutions to variation problems can be extended to satisfy Kurpradze-Sommerfeld radiation condition and Navier equations in upper half-space. In section 5, the Fr$\acute{\rm e}$chet derivatives for Dirichlet problem and Neumann problem are given respectively and a local stability result is given for Dirichlet problem. Conclusion is given in Section 6.

\section{Problem formulation}
\label{sec:sample2}
Consider a time-harmonic plane wave incident on a cavity, which is shown in Figure 1, denoted by $D$ with boundary $\Gamma\cup S$, where $\Gamma \subset \{x_2=0\}$ is bounded and $S$ is assumed Lipschitz continuous. It is noted that $S$ is not necessary a graph of some function. In addition, denote $\Gamma^c=\{x_2=0\}\backslash\Gamma$. The medium is assumed to be homogeneous.
\begin{figure}
    \centering
    \includegraphics{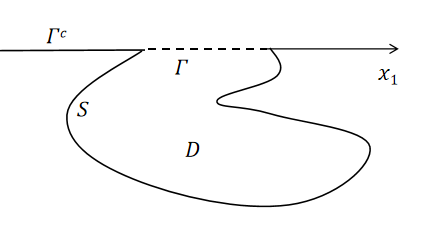}
    \caption{The problem geometry}
    \label{fig:my_label}
\end{figure}
Suppose the incident time-harmonic plane wave takes the form
\[
u_i=c_p\Vec{d}\,\text{exp}(i k_p(x_1\sin{\theta}-x_2\cos{\theta}))+c_s\Vec{d}^{\bot}\text{exp}(i k_s(x_1\sin{\theta}-x_2\cos{\theta})),
\]where
\[ c_p,c_s\in\mathbb{C},\quad
\vec{d}=(\sin{\theta},-\cos{\theta}),\quad \Vec{d}^{\bot}=(\cos{\theta},\sin{\theta}),\quad
\theta\in(-\pi/2,\pi/2).
\]
Define the compressional and shear wave numbers by \[
 k_p:=\omega/\sqrt{2\mu+\lambda},\quad k_s:=\omega/\sqrt{\mu}.\]
For simplicity, throughout this paper, it assumes  $c_p=k_p$, $c_s=0$. Consider total wave field $u$ which satisfies the following Navier equations with Lam$\acute{\rm e}$ constants $\lambda>0$, $\mu>0$ and frequency $\omega>0$
\begin{gather} \label{eq2.1}
    \mu\Delta u+(\mu+\lambda)\nabla(\nabla\cdot u)+\omega^2u=0 \quad {\rm in}\quad D\cup\{x_2>0\}.
\end{gather}
For convenience, denote by $\Delta^*u= \mu\Delta u+(\mu+\lambda)\nabla(\nabla\cdot u)$.
This paper considers two kinds of boundary conditions, i.e.,  Dirichlet condition
\begin{equation} \label{eq2.2}
 u=0 \quad {\rm on}\quad S\cup\Gamma^c,
\end{equation}
and Neumann condition\begin{equation} \label{eq2.3}
Tu=0 \quad\text{on}\quad S\cup\Gamma^c,
\end{equation}
where the differential operator is defined $Tu:=\mu\partial_nu+(\lambda+\mu)\Vec{n} {\rm div}\,u$.
Consider the Helmholtz decomposition for $u$, i.e.,
\begin{equation} \label{eq2.4}
    u=-i({\rm grad}\,\phi+{\overrightarrow{\rm  curl}}\,\psi),
    \end{equation}
    with
    \begin{equation} \label{eq2.5}
    \quad \phi =-\frac{i}{k^2_p}{\rm div}\,u, \text{ and }\psi=\frac{i}{k^2_s}{\rm curl}\,u,
\end{equation}where $\overrightarrow{\text{curl}}\, u=(\partial_2 u_1, -\partial_1 u_2)^\top$ and $\text{curl}\,u=\partial_1 u_2-\partial_2 u_1$. 
Then, $\phi$ and $\psi$ satisfy the homogeneous Helmholtz equations
 \begin{equation} \label{eq2.6}
     (\Delta+k^2_p)\phi=0\quad {\rm and}\quad (\Delta+k^2_s)\psi=0,\quad x_2>0.
 \end{equation}
The total wave field consists of incident wave $u_i$, reflected wave $u_r$ and scattering wave $u_s$, i.e., $u=u_i+u_r+u_s$. Here, the reflected wave $u_r$ satisfies
\begin{align}
 \Delta^*u_r +\omega^2 u_r=0 \quad {\rm in}\quad \{x_2>0\},\notag\\
   u_r+u_i=0\quad \text{on}\quad\{x_2=0\}, \label{eq2.7} \\
   \text{or}\quad T(u_i+u_r)=0 \quad\text{on}\quad\{x_2=0\} \label{eq2.8}.
\end{align}
It is easy to know that the reflected wave exists and can be expressed by the following functions \begin{align*}    
u_1&=\left(
    \alpha,\beta  
\right)^\top \text{exp}(i(\alpha x_1-\beta x_2)),\\
u_2&=\left(
    \alpha, -\beta  
\right)^\top \text{exp}(i(\alpha x_1+\beta x_2)),\\
u_3&=\left(
    \eta, -\alpha  
\right)^\top \text{exp}(i(\alpha x_1+\eta x_2)),
\end{align*}
with $\alpha=k_p\sin{\theta},\,\beta=k_p\cos{\theta}$ and $\eta=\sqrt{k_s^2-\alpha^2}$ which are linearly independent solutions to Navier equations. Combing the Dirichlet boundary condition gives \[ 
u_{r,1}=u_1-\frac{\alpha^2-\beta\eta}{\alpha^2+\beta\eta}u_2-\frac{2\alpha\beta}{\alpha^2+\beta\eta}u_3.
\]
And the Neumann boundary condition \eqref{eq2.8}  yields \begin{align*}
u_{r,2} =&\frac{(\lambda+\mu)k_p^2+\mu\beta^2\eta+\mu\alpha^2\beta}{2\mu\alpha^2\beta}u_1 \\ &+\frac{-(\lambda+\mu)k_p^2-\mu\beta^2\eta+\mu\alpha^2\beta}{2\mu\alpha^2\beta}u_2 +\frac{(\mu+\lambda)k^2_p+\mu\beta^2}{\mu\alpha\eta}u_3.
\end{align*}
For the scattering wave $u_s$, it is required to satisfy the following half-plane Kupradze-Sommerfeld radiation condition,
\begin{equation} \label{eq2.9}
    \lim_{r \to 0}r^{1/2}(\partial_r \phi(x)-i k_p\phi(x))=0 \quad {\rm in}\quad \mathbb{R}^+ ,
    \end{equation} and\begin{equation} \label{eq2.10}
         \lim_{r \to 0}r^{1/2}(\partial_r \psi(x)-i k_s\psi(x))=0 \quad {\rm in}\quad \mathbb{R}^+,
    \end{equation}
    with $r=|x|$ for $\phi=-\frac{i}{k^2_p} {\rm div} u_s$, $\psi=\frac{i}{k^2_s} {\rm curl} u_s$.

Now we can state the following two problems respectively.
  
\emph{ Dirichlet problem}: Find $u$ which satisfies \eqref{eq2.1}-\eqref{eq2.2}, where $u_s$ satisfies Kupradze-Sommerfeld radiation condition \eqref{eq2.9}-\eqref{eq2.10}.

\emph{ Neumann problem}: Find $u$ which satisfies \eqref{eq2.1} and \eqref{eq2.3}, where $u_s$ satisfies Kupradze-Sommerfeld radiation condition \eqref{eq2.9}-\eqref{eq2.10}.

To investigate the above two problems in an unbounded domain, it is necessary to introduce transparent boundary conditions (TBC) to reduce the model problems into a bounded domain.  

By Helmholtz decomposition, we have $u_s$ satisfying \eqref{eq2.4} with corresponding $\phi$, $\psi$ satisfying \eqref{eq2.5}-\eqref{eq2.6}.
 Applying the Fourier transform to \eqref{eq2.6} with respect to $x_1$ and using radiation condition gives 
 \begin{equation*}
     \hat{\phi}=P(\xi)\,\text{exp}(i x_2\gamma_p(\xi)),\quad
     \hat{\psi}=S(\xi)\,\text{exp}(i x_2\gamma_s(\xi)),
 \end{equation*}
 with \[\gamma_p(\xi)=\sqrt{k^2_p-\xi^2},\quad \gamma_s(\xi)=\sqrt{k^2_s-\xi^2}.\]
  Denote by $\hat{u}(\xi)=\mathcal{F}u(\xi)$ the Fourier transformation of $u$ with respect to $x_1$.
 Then the Fourier transformation of $u$ is given by \begin{equation} \label{eq2.11}
  \left(
\begin{array}{cc}
    \hat{u}_{s,1}   \\
    \hat{u}_{s.2}  
\end{array}\right)=\left(\begin{array}{cc}
    \xi & -i\partial_2 \\
    -i\partial_2 & -\xi
\end{array}\right)\left(\begin{array}{cc}
   P(\xi)\, \text{exp}(i x_2\gamma_p)  \\
   S(\xi)\, \text{exp}(i x_2\gamma_s)  
\end{array}\right).
\end{equation}
 Let $x_2=0$,
\begin{equation} \label{eq2.12}
    \hat{u}_s(\xi,0)=\left(\begin{array}{cc}
    \hat{u}_{s,1}(\xi,0)\\
    \hat{u}_{s,2}(\xi,0)\\
    \end{array}\right)=\left(
    \begin{array}{cc}
         \xi &\gamma_s  \\
         \gamma_p& -\xi
    \end{array}\right)\left(\begin{array}{cc}
         P(\xi)\\
         S(\xi)
    \end{array}\right),
\end{equation}
which implies 
\begin{equation} \label{eq2.13}
    \left(\begin{array}{cc}
         P(\xi)\\
         S(\xi)
    \end{array}\right)=\frac{1}{\xi^2+\gamma_p\gamma_s}\left(
    \begin{array}{cc}
         \xi &\gamma_s  \\
         \gamma_p& -\xi
    \end{array}\right)\left(\begin{array}{cc}
    \hat{u}_{s,1}(\xi,0)\\
    \hat{u}_{s,2}(\xi,0)\\
    \end{array}\right).
\end{equation}
Inserting \eqref{eq2.13} into \eqref{eq2.11} arrives at the following representation for $u_s$:
\begin{equation} \label{eq2.14}
    u_s=\frac{1}{\sqrt{2\pi}}\int_\mathbb{R}( \text{exp}(i x_2\gamma_p(\xi))M_p(\xi)+\text{exp}(i x_2\gamma_s(\xi))M_s(\xi)) \hat{u}_s(\xi,0)\text{exp}(ix_1\xi)\,\mathrm{d}\xi
\end{equation}
in $\{x_2>0\}$ with\[
    M_p(\xi)=\frac{1}{\xi^2+\gamma_p\gamma_s}\left(
    \begin{array}{cc}
         \xi^2 &\gamma_s\xi  \\
         \gamma_p\xi& \gamma_p\gamma_s
    \end{array}\right)
     ,\quad M_s(\xi)=\frac{1}{\xi^2+\gamma_p\gamma_s}\left(
    \begin{array}{cc}
         \gamma_p\gamma_s &-\gamma_s\xi  \\
         -\gamma_p\xi& \xi^2
    \end{array}\right).
\]
Recalling the definition of differential operator $T$, we have \begin{equation} \label{eq2.15}
    Tu_s:=\mu\partial_nu_s+(\lambda+\mu)\Vec{n} {\rm div}\,u_s \quad {\rm on}\quad\Gamma.
\end{equation}
Combing (2.14)-(2.15) and direct calculation implies that
\begin{equation} \label{eq2.16}
    Tu_s=\frac{1}{\sqrt{2\pi}}\int_\mathbb{R}M(\xi)\hat{u}_s(\xi,0)\text{exp}(ix\xi)\,\mathrm{d}\xi,
\end{equation}
where \begin{equation} \label{eq2.17}
  M(\xi)=  \frac{i}{\xi^2+\gamma_p\gamma_s}\left(
    \begin{array}{cc}
         \omega^2\gamma_p &-\xi\omega^2+\xi\mu(\xi^2+\gamma_p\gamma_s)  \\
         \xi\omega^2-\xi\mu(\xi^2+\gamma_p\gamma_s)& \omega^2\gamma_s
    \end{array}\right).
\end{equation}
Then the Dirichlet to Neumann operator can be defined by \begin{equation*}
\mathcal{T}f:=\mathcal{F}^{-1}(M\hat{f}),\quad f\in H^{1/2}(\mathbb{R})^2.
\end{equation*}
So $Tu_s=\mathcal{T}u_s$ on $\Gamma$. Noting that $u_i+u_r=0$ on $\{x_2=0\}$ and $u=0$ on $\Gamma^c$, then TBC for Dirichlet problem can be given by
\begin{equation*}
    Tu=\mathcal{T}u+g \quad {\rm on}\quad\Gamma ,\quad \text{with}\quad g=T(u_i+u_r).
    \end{equation*}
By TBC, Dirichlet problem can be reformulated as
\begin{align*}
\Delta^*u+\omega^2u=0\quad {\rm in}\quad D, \\
u=0\quad {\rm on}\quad S, \\
Tu=\mathcal{T}u+g \quad {\rm on}\quad\Gamma.
\end{align*}
Denote $V=H^1_S(D)^2=\{u\in H^1(D)^2:u=0\, \mathrm{ on }\, S\}$.
Then for $u,\,v \in V$, the Betti formula gives
\begin{equation*}
    0=-\int_D (\Delta^*+\omega^2)u\cdot\bar{v}\,\mathrm{d}x= \int_D \mathcal{E}(u,\bar{v})-\omega^2 u\cdot\bar{v}\,\mathrm{d}x-\int_{\Gamma}(\mathcal{T}u+g)\cdot\bar{v}\,\mathrm{d}s,
\end{equation*}
where\begin{equation} \label{eq2.18}
    \mathcal{E}(u,v)=\mu(\nabla u_1\cdot\nabla v_1+\nabla u_2\cdot \nabla v_2)+(\lambda+\mu)(\nabla\cdot u)(\nabla\cdot v).
\end{equation}
Define a sesquilinear form $B\, :V\times V\to \mathbb{C}$ as
\begin{equation*}
    B(u,v)=
\int_D \mathcal{E}(u,\bar{v})-\omega^2 u\cdot\bar{v}\,\mathrm{d}x-\int_{\Gamma}\mathcal{T}u\cdot\bar{v}\,\mathrm{d}s.
\end{equation*}
Now we obtain variation formula of Dirichlet boundary value problem.

 \emph{Variation problem 1}: find $u\in V$ such that $B(u,v)=(g,v)_\Gamma ,\,\forall  v\in V$.
 
It is similar to give TBC and variation formula corresponding to Neumann problem. Note that in this case, it is not $u_s$ but $Tu_s$, which has compact support on $\Gamma$. So $\hat{u}_s$ may not be integrable on $\{x_2=0\}$. Thus, \eqref{eq2.13} and \eqref{eq2.16} may not hold in this case. However, $u_s$ can be represented by $\widehat{Tu_s}(\xi,0)$ similarly. According to \eqref{eq2.14}, we obtain \begin{equation} \label{eq2.19}
\widehat{T u_s}(\xi,x_2)=\left( \begin{array}{cc}
    \mu\partial_2 & 0 \\
    i(\lambda+\mu)\xi & (\lambda+2\mu)\partial_2
\end{array}\right)\left(\begin{array}{cc}
    \hat{u}_{s,1}  \\
    \hat{u}_{s,2} 
\end{array}\right).
\end{equation}
 Combing \eqref{eq2.11} and \eqref{eq2.19} yields
 \[ \widehat{T u_s}(\xi,x_2)=\left( \begin{array}{cc}
    \mu\partial_2 & 0 \\
    i(\lambda+\mu)\xi & (\lambda+2\mu)\partial_y
\end{array}\right) \left(\begin{array}{cc}
    \xi & -i\partial_2 \\
    -i\partial_2 & -\xi
\end{array}\right)\left(\begin{array}{cc}
   P(\xi)\, \text{exp}(i x_2\gamma_p)  \\
   S(\xi)\, \text{exp}(i x_2\gamma_s)  
\end{array}\right).
 \]
 Take $x_2=0$, then $P,\,S$ can be represented by $\widehat{Tu_s}(\xi,0)$. Insert the representation in \eqref{eq2.11}, we have\begin{equation*}
\hat{u}_s(\xi,x_2)=\Big(\text{exp}\big(i x_2\gamma_p(\xi)\big)M_p(\xi)+\text{exp}\big(i x_2\gamma_s(\xi)\big)M_s(\xi)\Big)M^{-1} \widehat{T u_s}(\xi,0).
 \end{equation*}
 Taking inverse Fourier transformation and $x_2=0$ gives \[
u_s(\xi,0)=\mathcal{F}^{-1}(M^{-1}\widehat{u_s}(\xi,0)).
 \]
 Hence NtD operator $\varLambda$ can be defined by
\begin{equation*}\varLambda(f):=\mathcal{F}^{-1}(M^{-1}\hat{f}),\quad f\in H^{-1/2}(\mathbb{R})^2,\end{equation*}
 which implies
 \[\varLambda(Tu_s)=u_s\quad \text{on}\quad \Gamma.\]
Denote $h=u_r+u_i$ and $g=-(\Delta^*+\omega^2)h$. Then the Neumann problem can be reformulated as 
\begin{align}
    \Delta^*u_s+\omega^2u_s=0\quad {\rm in}\quad D, \notag \\
Tu_s=-Th\quad {\rm on}\quad S,  \\
u_s=\varLambda(Tu_s) \quad {\rm on}\quad \Gamma. \notag 
\end{align}
Define the function space
\[W=\{u\in H^1(D)^2:Tu\in H^{-1/2}_0(\Gamma)^2,u=\varLambda(Tu)\, \text{on}\,\Gamma\},\]
where\[ H^{-1/2}_0(\Gamma)^2=\{u\in H^{-1/2}(\Gamma)^2: \varLambda(\tilde{u})\in H^{1/2}(K)^2 ,\,{\forall}  K\subset\subset \mathbb{R}^2\},\] 
and $\tilde{u}$ is zero extension of $u$.
In fact, it is clear that $H^{-1/2}_0(\Gamma)^2$ is the dual space of $H^{1/2}(\Gamma)^2$ in \cite{r15}.
The sesquilinear form $\mathcal{B}$: $W\times W \rightarrow \mathbb{C}$ is defined by \begin{equation*}\mathcal{B}(v,w)=
\int_D \mathcal{E}(v,\bar{w})-\omega^2 v\cdot\bar{w}\,\mathrm{d}x-\int_{\Gamma}Tu\cdot \overline{\varLambda(Tw)} \,\mathrm{d}s.
\end{equation*}
Now we can give the following variation formula by the Betti formula.

\emph{Variation problem 2}: find $u_s\in W$ such that \[\mathcal{B}(u_s,w)=\int_S -Th \cdot \bar{w}\,\text{d}s-\int_D g\cdot \bar{w} \,\text{d}x,\quad\forall w\in W.\]

Next two sections consider existence and uniqueness of solutions to the above two variation problems. Moreover,  they can be extended to $\{x_2>0\}$ satisfying Kurpradze-Sommerfeld radiation condition.
\section{Existence and uniqueness for Dirichlet problem}
\label{sec:sample3}

Since this problem is in a bounded domain with Lipschitz continuous boundary, we have the well-known compact embedding result that $H^1(D)$ is a compact subset of $ L^2(D)$. Hence uniqueness of variation problem and Fredholm alternative yield existence directly. To this end, we show that the continuity of DtN operator results in the continuity of the sesquilinear form $B$, which is stated in the following lemma with the proof omitted.
\begin{Le}
    The DtN operator $\mathcal{T}$ is a continuous linear operator from $H^{1/2}_0(\Gamma)^2$ to $H^{-1/2}(\Gamma)^2$ .
\end{Le}
 In order to use Fredholm alternative, it needs to show the sesquilinear form satisfies the  G$\mathring{\rm a}$rding inequality.
Take real part of $B$, \[\Re\,B(u,u)=\int_D \mathcal{E}(u,\bar{u})-\omega^2 u\cdot\bar{u}\,\mathrm{d}x-\Re\int_{\Gamma}\mathcal{T}u\cdot\bar{u}\,\mathrm{d}s.\]
Considering \eqref{eq2.18} gives \[\int_D\mathcal{E}(u,\bar{u})\,\mathrm{d}x=\mu\|\nabla u\|^2_{L^2(D)^2}+(\lambda+\mu)\|\nabla \cdot u\|^2_{L^2(D)^2}\ge \mu\|\nabla u\|^2_{L^2(D)^2}.\]
It turns out that \begin{equation}\Re\,B(u,u)\ge \mu\|\nabla u\|^2_{L^2(D)^2}-\omega^2\|u\|^2_{L^2(D)^2}-\Re\,\left\langle\mathcal{T}u,u\right\rangle_{\Gamma}.\end{equation}
By Plancherel identity,\[-\Re\,\left\langle\mathcal{T}u,u\right\rangle_{\Gamma}=-\left\langle \Re M\,\mathcal{F}u, \mathcal{F}u \right\rangle_\Gamma,\]
where $\Re\,M=(M+\bar{M}^{\top})/2$. However, the sign of $-\Re\,\left\langle\mathcal{T}u,u\right\rangle_{\Gamma}$ is undeterminied because it lacks the result that $-\Re \,M(\xi)$ is positive definite or negative definite. Fortunately, for sufficiently large $|\xi|$, we know $-\Re\,M(\xi)$ is positive definite, given in Lemma 3.2.
\begin{Le}
    $-\Re \,M(\xi)>0$ for $|\xi|>k_s$.
\end{Le}
The proof is also omitted here as it is trivial compared as the results in \cite{r7}.
Noting that \[\left\langle\mathcal{F}u, \mathcal{F}u\right\rangle=\int_{|\xi|>k_s}M(\xi)|\hat{u}|^2\mathrm{d}\,\xi+\int_{|\xi| \le k_s}M(\xi)|\hat{u}|^2\mathrm{d}\,\xi,\] and using Lemma 3.2, we have\[-\left\langle\Re\,\mathcal{F}u, \mathcal{F}u\right\rangle \ge-\int_{|\xi| \le k_s}\Re\,M(\xi)|\hat{u}|^2\mathrm{d}\,\xi.\]Therefore, it suffices to estimate the intergal in $\{|\xi| \le k_s\}$. The following inequality is similar as the three-dimensions case in \cite{r8}.
\begin{Le}
    For any $u\in V$, we have $\Re\,B(u,u)>C_1\|\nabla u\|^2_{L^2(D)^2}-C_2\|u\|^2_{L^2(D)^2}$, where $C_1,\,C_2$ are positive constants which depend on $\lambda,\,\mu,\,\omega$.
    
\end{Le}
Next the uniqueness and existence of the variation problem can be derived by Lemma 3.3 and the Fredholm alternative.
\begin{Th}
    Variation problem 1 admits a unique solution $u\in V$.
\end{Th}
\begin{proof}
    Assume $g=0$. We have\[0=\Im\,\langle g,u\rangle_{\Gamma}=-\Im\,B(u,u)=\Im\,\left\langle\mathcal{T}u,u\right\rangle_{\Gamma}=\Im\,\left
    \langle M \hat{u},\hat{u}\right\rangle_\mathbb{R}.\]
    Combing (2.12) and (2.17) gives
    \[
   \left
    \langle M \hat{u},\hat{u}\right\rangle_\mathbb{R}= 
     \int_{\mathbb{R}}
    i\left(\begin{array}{cc}
         P  \\
         S 
    \end{array}\right)^\top\left(\begin{array}{cc}
       \mu\xi\gamma_p  & \omega^2-\mu\xi^2 \\
        \omega^2-\mu\xi^2 & -\mu\xi\gamma_s
    \end{array}\right)\left(\begin{array}{cc}
        \xi & \bar{\gamma_s} \\
        \bar{\gamma_p} & -\xi
    \end{array}\right)\left(\begin{array}{cc}
         \bar{P}  \\
         \bar{S} 
    \end{array}\right)\,\mathrm{d}\xi.
    \]
    Denote the matrix by \[
A=\left(\begin{array}{cc}
       \mu\xi\gamma_p  & \omega^2-\mu\xi^2 \\
        \omega^2-\mu\xi^2 & -\mu\xi\gamma_s
    \end{array}\right)\left(\begin{array}{cc}
        \xi & \bar{\gamma_s} \\
        \bar{\gamma_p} & -\xi
    \end{array}\right).
    \] 
Taking the real part implies\[
\Re\,A=\left\{\begin{array}{ccc}
    0 & \text{ if  } |\xi|>k_s, \\
    \omega^2 \text{diag}(0,\gamma_s) & \qquad\text{ if }\, k_p<|\xi| \le k_s, \\
    \omega^2 \text{diag}(\gamma_p,\gamma_s) & \text{ if }\, |\xi| \le k_p.
\end{array}\right.
\]
So we have\begin{equation}\Im\langle\mathcal{T}u,u\rangle_{\Gamma}=\omega^2\left(\int_{|\xi| \le k_p}\gamma_p(\xi)|P(\xi)|^2\,\mathrm{d}\xi+\int_{|\xi| \le k_s}\gamma_s(\xi)|S(\xi)|^2\,\mathrm{d}\xi\right)=0.\end{equation}
Hence
\[P(\xi)=S(\xi)=0\quad\text{for}\quad|\xi| \le k_p,\]
which implies
\[\hat{u}(\xi,0)=0\quad\text{for}\quad|\xi| \le k_p.\]
Since $u$ has compact support on $\{y=0\}$, $\hat{u}(\xi,0)$ is analytic with respect to $\xi$. By unique continuation, $\hat{u}(\xi,0)=0$ for $\xi\,\in\mathbb{R}$. So  $u(x,0)=0$ and $Tu=\mathcal{T}u=0$ on $\Gamma$. By Holmgren's uniqueness theorem,  $u=0$ in $D$ which implies uniqueness. Finally, combing Lemma 3.3 and Fredholm alternative yields existence. 
\end{proof}
\emph{Remark}. Theorem 3.1 shows that there admits a unique solution $u$ to \emph{Variation problem 1} in bounded domain $D$. We can extend $u_s=u-u_i-u_r$ to $\{x_2>0\}$ by (2.14), though it may not satisfy Kupradze-Sommerfeld radiation condition. In fact, the condition that $u_s$ satisfies (2.14) is weaker than Kupradze-Sommerfeld radiation condition of half-plane \eqref{eq2.9}-\eqref{eq2.10} (see \cite{r1}). But we can extend $u_s$ by
\begin{equation}\label{eq3.3}
u_s=\int_{\Gamma}T_z G_D(x,y)u_s(y)\,\text{d}\,s(y),\quad x\in\{x_2>0\},
\end{equation} where $G_D(x,y)$ is the half-space Green tensor for Dirichlet problem. According to the property of Green tensor we can verify that the corresponding $u$ is the solution to Drichlet problem in $D \cup \{x_2 \ge 0\}$ (see \cite{r1}). Then by representation theorem we know that any solution to Dirichlet problem in $D \cup \{x_2 \ge 0\}$ satisfies \eqref{eq3.3}, which means the extension is unique.
\section{Existence and uniqueness for Neumann problem}
Recalling the two variation problems in Section 2, we can find that the main difference between Dirichlet problem and Neumann one is that $M$ in \eqref{eq2.17} is replaced by $M^{-1}$, i.e.,
\begin{equation} \label{eq4.1}
M^{-1}=\frac{\rho(\xi)}{id(\xi)}\left(\begin{array}{cc}
        \omega^2\gamma_s & \xi\omega^2-\xi\mu\rho \\
        -\xi\omega^2+\xi\mu\rho & \omega^2\gamma_p
    \end{array}
    \right),
\end{equation}
where\[
\rho=\xi^2+\gamma_s\gamma_p,\,d=\omega^4\gamma_s\gamma_p+(\xi\omega^2-\xi\mu\rho)^2.
\]
 This section shows that $M^{-1}$ has similar properties to $M$ though it is more complicated. Before proceed to the uniqueness, we give the following lemma to show that $\varLambda$ is continuous, which implies the sesquilinear form is also continuous.
\begin{Le}
    The NtD operator $\varLambda$ is a bounded linear operator from $H^{-1/2}_0(\Gamma)^2$ to $H^{1/2}(\Gamma)^2$.
\end{Le}
\begin{proof}\begin{equation} \label{eq4.2}
    \|\varLambda\|^2=\sup\limits_{u\in W,\,u\not=0}\frac{\|\varLambda u\|^2_{H^{1/2}(\Gamma)^2}}{\|u\|^2_{H^{-1/2}(\Gamma)^2}}=\sup\limits_{u\in W,\,u\not=0}\frac{\int_{\mathbb{R}}|M^{-1}\hat{u}|^2(1+\xi^2)^{1/2}}{\int_{\mathbb{R}}|\hat{u}|^2(1+\xi^2)^{-1/2}}.
    \end{equation}
   When $|\xi|\to+\infty$, it is easy to calculate \[\gamma_p\sim i|\xi|,\,\gamma_s\sim i|\xi|,\,\text{and}\,(\xi^2+\gamma_s\gamma_p)\sim\frac{k^2_p+k^2_s}{2}.
    \]
    It turns out that
    \begin{equation} \label{eq4.3}
    	\left[ \xi\omega^2-\xi\mu(\xi^2+\gamma_s\gamma_p)\right] \sim\xi\left(\omega^2-\frac{k_s^2+k_p^2}{2}\right),\end{equation}
    and\begin{equation} \label{eq4.4}
    	\left[\omega^4\gamma_s\gamma_p+(\xi\omega^2-\xi\mu(\xi^2+\gamma_s\gamma_p))^2\right] \sim\xi^2\left[ \omega^4-\left(\omega^2-\mu\frac{k^2_s+k^2_p}{2}\right)^2\right]. \end{equation}
     Combing \eqref{eq4.1} and \eqref{eq4.3}-\eqref{eq4.4} gives estimate of $\|M^{-1}\|$,
    \[\|M^{-1}\|\le C(1+\xi^2)^{-1}.\]
    Inserting it into \eqref{eq4.2} yields $\|\varLambda\|\le C$, where $C$ depends on $\omega,\,\mu$ and $\lambda$.
\end{proof}
Then we consider $-\Re\,M^{-1}$ similarly as Section 2. 
\begin{Le}
    $-\Re\,M^{-1}(\xi)>0$ for sufficiently large $|\xi|$.
\end{Le}
\begin{proof}
    Taking real part of $M^{-1}$ for $|\xi|>k_s$, we have
    \[\Re\,M^{-1}=\frac{\rho(\xi)}{d(\xi)}\left(\begin{array}{cc}
        \omega^2|\gamma_s| & -i(\xi\omega^2-\xi\mu\rho) \\
        i(\xi\omega^2-\xi\mu\rho) & \omega^2|\gamma_p|
    \end{array}
    \right),\]
    with
    \[\rho(\xi)=\xi^2-|\gamma_s||\gamma_p|>0,\]
    and
    \[d(\xi)=-\omega^4|\gamma_s||\gamma_p|+(\xi\omega^2-\xi\mu\rho)^2.\]
    In order to prove $-\Re M^{-1}(\xi)>0$, it suffices to verify $\frac{\rho(\xi)}{d(\xi)}\omega^2|\gamma_s|<0$ and $\det{-\Re M^{-1}}>0$.
    Recall when $|\xi|\to\infty$, $\gamma_s\sim i|\xi|,\,\gamma_p\sim i|\xi|$ and $\rho\sim\frac{k^2_s+k^2_p}{2}$. It turns out that
    \[
    d\sim\xi^2\left(\left(\omega^2-\mu\frac{k^2_s+k^2_p}{2}\right)^2-\omega^4\right)=\omega^2\left(\frac{\mu}{4\mu+2\lambda}-\frac{3}{2}\right)<0,\] which means $\frac{\rho(\xi)} {d(\xi)}\omega^2|\gamma_s|<0$ for sufficiently large $|\xi|$. Direct calculation gives
    \[\det(-\Re\,M^{-1})=\frac{\rho^2}{d^2}(\omega^4|\gamma_s||\gamma_p|-(\xi\omega^2-\xi\mu\rho)^2)>0,\] for sufficiently large $|\xi|$.
    So $-\Re\,M^{-1}$ is positive definite for sufficiently large $|\xi|$. 
\end{proof}
Using the above lemma, we can estimate $\Re\,\mathcal{B}(u,u)$ and then get the following inequality.
\begin{Le}
 For any $u\in W$, we have $\Re\,\mathcal{B}(u,u)>C_1\|\nabla u\|^2_{L^2(D)^2}-C_2\|u\|^2_{L^2(D)^2}$, where $C_1,\,C_2$ are positive constant and they both depend on $\lambda,\,\mu$ and $\omega$.
\end{Le}
\begin{proof}
Similar to Section 3, for $u\in W$, we have
    \begin{equation} \label{eq4.5} 
    	\Re\,\mathcal{B}(u,u)\ge\mu\|\nabla u\|^2_{L^2(D)^2}-\omega^2\|u\|^2_{L^2(D)^2}-\Re\,\langle Tu,\varLambda (Tu)\rangle_\Gamma.\end{equation}
    By Lemma 4.2, there exists $k_0>0$ such that $-\Re\,M^{-1}(\xi)>0$ for $|\xi|>k_0$.
    So it can be deduced that \begin{equation} \label{eq4.6}
-\Re\,\langle Tu,\varLambda(Tu)\rangle_\Gamma\ge-\Re\,\int_{|\xi|\le k_0}\widehat{Tu}\cdot\overline{M^{-1}\widehat{Tu}}\,\mathrm{d}\,\xi\ge -C\int_{|\xi|\le k_0}\|M^{-1}\| |\widehat{Tu}|^2\,\mathrm{d}\,\xi.
    \end{equation}
    Noting for $|\xi| \le k_0$,\[
\|M^{-1}\|\le C(1+\xi^2)^{-1/2}\le C(1+\xi^2)^{-1}.
    \]
    So we have\begin{equation} \label{eq4.7}
\int_{|\xi| \le k_0}\|M^{-1}\| |\widehat{Tu}|^2 \,\mathrm{d}\,\xi\le C\int_{|\xi| \le k_0}(1+\xi^2)^{-1}|\widehat{Tu}|^2\,\mathrm{d}\,\xi\le C\|Tu\|^2_{H^{-1}(\Gamma).}
    \end{equation}
    By the trace theorem and the interpolation inequality, we obtain
    \begin{equation}\label{eq4.8}
    	\|Tu\|^2_{H^{-1}(\Gamma)^2}\le C \|u\|^2_{L^2(\Gamma)^2}\le C\|u\|^2_{H^{1/2}(D)^2}\le \varepsilon\|u\|^2_{H^1(D)^2}+C(\varepsilon)\|u\|^2_{L^2(D)^2}.\end{equation}
    Here $C>0$ are different constants depending on $\mu,\,\lambda$ and $\omega$, $\varepsilon>0$ is sufficiently small and $C(\varepsilon)>0$ is dependent on $\mu,\,\lambda,\,\omega$ and $\varepsilon$.
    By combining \eqref{eq4.5}-\eqref{eq4.8}, we obtain the inequality
    \[
\Re\,\mathcal{B}(u,u)\ge\mu\|\nabla u\|^2_{L^2(D)^2}-\omega^2\|u\|^2_{L^2(D)^2}-C\varepsilon\|u\|^2_{H^1(D)^2}-C(\varepsilon)\|u\|^2_{L^2(D)^2}.\] As $\varepsilon$ is sufficiently small, it turns out that
\[\Re\,\mathcal{B}(u,u)
\ge C_1\|\nabla u\|^2_{L^2(D)^2}-C_2\|u\|^2_{L^2(D)^2}. \]
\end{proof}
Next, we can proceed to the uniqueness for the Neumann problem. Though $M^{-1}$ is more complicated than $M$, it is difficult to directly give the representation of the imaginary part of the sesquilinear form like (3.2). But the following proof shows that it is not necessary to get exact representation of the imaginary part of the sesquilinear form.
\begin{Th}
     Variation problem 2 admits a unique solution $u\in W$.
\end{Th}
\begin{proof}
    Similar to Section 3, we will prove the uniqueness, which yields existence. Assume $h=0$ (i.e., $u=u_s$) which implies $g=0$. Taking imaginary part of variation formula gives \[
\Im\,\mathcal{B}(u,u)=-\Im\,\langle Tu, \varLambda(Tu)\rangle_{\Gamma}=-\Im\,\langle\widehat{Tu}, M^{-1}\widehat{Tu}\rangle_{\Gamma}=\langle\Im\,M^{-1}
\widehat{Tu}, \widehat{Tu}\rangle_{\Gamma}=0
    .\]
    For $|\xi|>k_s$, we have \[ \rho=\xi^2-|\gamma_s||\gamma_p|,\,d=-\omega^4|\gamma_s||\gamma_p|+(\xi\omega^2-\xi\mu\rho)^2.
    \]
    Then take imaginary part of \eqref{eq4.1} to get $\Im M^{-1}=0$.
    For $|\xi| \le k_p$, we have\[\rho=\xi^2+|\gamma_s|\gamma_p|>0,\,d=\omega^4|\gamma_s||\gamma_p|+(\xi\omega^2-\xi\mu\rho)^2>0.\] Simple calculation gives \[\Im M^{-1}=
-\frac{\rho(\xi)}{d(\xi)}\left(\begin{array}{cc}
        \omega^2|\gamma_s| & 0 \\
       0 & \omega^2|\gamma_p|
    \end{array}
    \right).
    \]
It turns out that $\Im M^{-1}<0$.
    For $k_p<|\xi| \le k_s$, we have\[
\rho=\xi^2+i|\gamma_s|\gamma_p|,\,d=\omega^4|\gamma_s||\gamma_p|i+(\xi\omega^2-\xi\mu\rho)^2. 
    \]
    Direct calculation gives
    \[
\Im M^{-1}=\left(\begin{array}{cc}
   a_{1}  & b_1+b_2 \\
   -b_1-b_2  & a_{2}
\end{array}\right),
    \]
    where\[
a_{1}=-\omega^6|\gamma_s|^3|\gamma_p|^2+2\omega^2|\gamma_p|^2|\gamma_s|^3\xi^2\mu\zeta-\xi^4\omega^2|\gamma_s|\zeta^2+\omega^2\xi^4\mu^2|\gamma_p|^2|\gamma_s|^3,
    \]
    \[
a_{2}=-\xi\omega^6|\gamma_p|^2|\gamma_s|+2\xi^4\omega^2\mu\zeta|\gamma_s||\gamma_p|^2+\omega^2|\gamma_p|^2|\gamma_s|\xi^2\zeta^2-\xi^2\mu^2|\gamma_p|^4|\gamma_s|^3,
    \]
    \[
b_1=i\xi^3\omega^4\zeta|\gamma_s||\gamma_p|-2i\xi^5\mu\zeta^2|\gamma_s||\gamma_p|-i\xi^3\zeta^3|\gamma_s||\gamma_p|+i\xi^3\mu^2\zeta|\gamma_s|^3|\gamma_p|^3,
    \]
    \[
b_2=i\omega^4\xi\mu|\gamma_p|^3|\gamma_s|^3-2i\xi^3\mu^2|\gamma_p|^3|\gamma_s|^3\zeta+i\xi^5\mu\zeta^2|\gamma_s||\gamma_p|-i\xi^5\mu^3|\gamma_p|^3|\gamma_s|^3,
    \]
    with\[
\zeta=\omega^2-\xi^2\mu.
    \]
    In order to prove $\Im M^{-1} \le 0$, it suffices to verify $a_1<0$ and $\det{\Im M^{-1}} =0$. By direct calculation,\[
a_1=(-|\gamma_s|^2|\gamma_p|^2\zeta^2-\xi^4\zeta^2)\omega^2|\gamma_s|<0,\]
and \[
\det{\Im M^{-1}}=a_1a_2+(b_1+b_2)^2=d_1|\gamma_s|^6|\gamma_p|^6+d_2|\gamma_s|^4|\gamma_p|^4+d_3|\gamma_s|^2|\gamma_p|^2,
    \]
    where\begin{align*}
d_1&=\omega^4\xi^2\mu^2(\omega^4-\xi^2\mu(2\omega^2-\xi^2\mu))-(\zeta\xi\mu\omega^2)^2=0,\\
d_2&=\xi^6\omega^4\mu^2(\omega^4-\xi^2\mu(2\omega^2-\xi^2\mu))+\xi^6\omega^4\mu^2\zeta^2-2\xi^6\mu^2\omega^4\zeta^2=0,\\   
d_3&=\xi^{10}\zeta^2\omega^4\mu^2-(\xi^3\omega^4\zeta-\xi^3\omega^2\zeta^2)^2=0.
    \end{align*}
 Therefore it turns out that $\det \Im M^{-1}=0$, and thus $\Im M^{-1} \le 0$.
    
    In summary,\[
\Im\,M^{-1}\left\{\begin{array}{ccc}
    =0 & \text{if}\, |\xi|>k_s, \\
    \le 0 & \text{if}\, k_p<|\xi| \le k_s, \\
    <0 & \text{if}\, |\xi| \le k_p.
\end{array}\right.
    \]
    Hence $\widehat{Tu}=0$ for $|\xi| \le k_p$. Similarly to Section 3, by unique continuation and Holmgren's uniqueness theorem, $u=0$ in $D$ which implies uniqueness. Combing Lemma 4.3 and Fredholm alternative yields existence.
    
\end{proof}
\emph{Remark}. Similarly to Section 2, we can extend $u_s$ to $\{x_2>0\}$ by\begin{equation}\label{eq4.9}
u_s=
-\int_{\Gamma} G_N(x,y)T u_s(y)\,\text{d}s(y)\quad x\in\{x_2>0\},
\end{equation}where $G_N(x,y)$ is the half-space Green tensor for Neumann problem. Obviously by representation theorem, this extension is unqiue. Hence we get the unique solution to Neumann problem in $D\cup \{x_2 \ge 0\}$.
\section{Inverse cavity problem}
In this section, an inverse cavity problem, i.e., to reconstruct the unknown cavity $S$ from $u|_\Gamma$, is considered by investigating 
the Fr$\acute{\rm e}$chet derivative of the solution operator. For Dirichlet problem, the Frech$\acute{\rm e}$t derivative implies a local stability result of inverse cavity problem directly. This can be deduced by the fact that when the Dirichlet data of Frech$\acute{\rm e}$t derivative on the  boundary of the cavity vanishes, the solution $u$ to original problem is zero. However, for Neumann case, the Frech$\acute{\rm e}$t derivative satisfies different boundary value condition, which results in major difference from the Dirichlet case. 
\subsection{Dirichlet case}

Let $S$ be $C^2$ and $h(x) \in C^2(S,\mathbb{R})^2$ satisfying $h(x)|_{S\cap\Gamma}=0$. Denote by $S_h=\{x+h(x):x\in S\}$ and $D_h$ the domain with boundary $S\cup \Gamma$. Similarly denote $V_h=\{u \in H^1(D_h)^2 : u=0, \,\text{on}\, S_h \}$. Define the solution operator $\mathcal{U}$ : $C^2(S,\mathbb{R}^2) \to H^{1/2}(\Gamma)^2$ by \[
\mathcal{U}(h)=u_h|_\Gamma,
\]where $u_h$ is the solution to variation problem for Dirichlet boundary condition, i.e., \begin{equation} \label{eq5.1}
B_h(u_h,v_h)=(g,v_h)_\Gamma, \quad \forall v_h \in V_h
\end{equation} with \begin{equation} \label{eq5.2}
B_h(u_h,v_h)=\int_{D_h}\mathcal{E}(u_h,\bar{v}_h)-\omega^2 u_h \cdot \bar{v} \,\mathrm{d}x-\int_\Gamma \mathcal{T}u \cdot \bar{v} \,\mathrm{d} s.
\end{equation} Extend $h(x)$ to $C^2(\bar{D},\mathbb{R}^2)$  and still denote it by $h$ such that $
h(x)|_\Gamma=0
$, which gives \begin{equation} \label{eq5.3}
\|h\|_{H^{1,\infty}(D)^2}\le C\|h\|_{H^{1,\infty}(S)^2}.
\end{equation}

 Denote $\mathcal{H}_h$ by $
\mathcal{H}_h(y)=y+h(y),\quad y \in D$. Obviously, it is invertible for sufficiently small $\|h\|_{1,\infty}$. Then taking the variable transform $x=\mathcal{H}_h(y)$ implies 
\begin{align*}
    B_h(u_h,v_h)=&\mu\int_{D} \sum_{j=1}^2\nabla\tilde{u}_{h_j}\mathcal{J}_{\mathcal{H}_h^{-1}}\mathcal{J}_{\mathcal{H}_h^{-1}}^\top\nabla \bar{\tilde{v}}_{h_j}\det{\mathcal{J}_{\mathcal{H}_h}}\,\text{d}y \\
    &+(\lambda+\mu)\int_{D} (\nabla\tilde{u}_h:\mathcal{J}_{\mathcal{H}_h^{-1}})(\nabla\bar{\tilde{v}}_h:\mathcal{J}_{\mathcal{H}_h^{-1}}^\top)\det{\mathcal{J}_{\mathcal{H}_h}}\,\text{d}y \\
    &-\omega^2\int_{D}\tilde{u}_h\cdot\bar{\tilde{v}}_h\det{\mathcal{J}_{\mathcal{H}_h}}\,\text{d}y -\int_{\Gamma}\mathcal{T}\tilde{u}_h\cdot\bar{\tilde{v}}_h \,\text{d}s(y) :=B^h(\Tilde{u}_h, \Tilde{v}_h),\end{align*}
where $\tilde{u}_h=u_h \circ \mathcal{H}_h \in V$,  $\tilde{v}_h=v_h \circ \mathcal{H}_h \in V$. Hence $B^h$ is a sesquilinear form on $V \times V$. Thus the variation formula \eqref{eq5.1} becomes
\begin{equation} \label{eq5.4}
    B^h( \Tilde{u}_h,v)=(g,v)_\Gamma,\quad \forall v \in V.
\end{equation} 
Next we proceed to derive the Fr$\acute{\rm e}$chet derivative of the solution operator $\mathcal{U}$. Let $\mathcal{U}'(0)$ : $C^2(S,\mathbb{R})^2 \to H^{1/2}(\Gamma)^2$ is the Fr$\acute{\rm e}$chet derivative of $\mathcal{U}$ at $0$. Notice that $C^2(S,\mathbb{R})^2$ is equipped with norm $\|\cdot\|_{H^{1,\infty}(S)^2}$. The following theorem shows for $h \in C^2(S,\mathbb{R})^2$, $\mathcal{U}'(0)h$ satisfies the homogeneous Navier equation in $D$, the homogeneous TBC on $\Gamma$ and an inhomogeneous Dirichlet condition \eqref{eq5.7} on $S$.
\begin{Th}
   Suppose $h \in C^2(S, \mathbb{R})^2$ satisfying $h=0$ on $S \cap \Gamma$. Let $u^*$ is the solution to the problem \begin{align} \label{eq5.5}
        (\Delta^*+\omega^2)u^*=0, \quad &{\rm in} \quad D,  \\ \label{eq5.6}
        Tu^*=\mathcal{T} u^*, \quad &{\rm on}\quad \Gamma,  \\ \label{eq5.7}
        u^*=-(n \cdot h) \partial_n u, \quad &{\rm on} \quad S.
    \end{align}
    Then $\mathcal{U}'(0)h=u^*|_\Gamma$.
\end{Th}
\begin{proof}
    Let $w_h=h \cdot \nabla u$. Then $h|_\Gamma=0$ implies $w_h|_\Gamma=0$. So $w_h+u^* \in V$ and $(w_h+u^*)|_\Gamma=u^*|_\Gamma$. By trace theorem, \begin{equation}\label{eq5.8} \|u_h-u-u^*\|_{H^{1/2}(\Gamma)^2} \le C\|\tilde{u}_h-u-u^*-w_h\|_{H^1(D)^2}
    \end{equation} with the constant $C>0$. For convenience, we denote $\|h\|_{1,\infty}=\|h\|_{H^{1,\infty}(D)^2}$. Then combining \eqref{eq5.3} and \eqref{eq5.8} yields we only need to prove\[
    \lim_{\|h\|_{1,\infty}\to 0} \frac{\|\tilde{u}_h-u-u^*-w_h\|_{H^1(D)^2}}{\|h\|_{1,\infty}}=0.
    \] For any $v \in V$,\[
    B(\tilde{u}_h-u-u^*-w_h,\phi)=B(\tilde{u}_h-u,v)-B(u^*+u,v).
    \] Firstly consider\[
    B(\tilde{u}_h-u,v)=B(\tilde{u}_h,v)-(g,v)_\Gamma=B(\tilde{u}_h,v)-B^h(\tilde{u}_h,v)=\sum_{i=1}^3B_i(\tilde{u}_h,v),
    \] where\begin{align}\label{eq5.9}
        B_1(\tilde{u}_h,v)&=\mu\int_{D} \sum_{j=1}^2\nabla \tilde{u}_{h,j}(I_2-\mathcal{J}_{\mathcal{H}^{-1}_h}\mathcal{J}_{\mathcal{H}^{-1}_h}^\top\det{\mathcal{J}_{\mathcal{H}_h}})\nabla \bar{v}_j\,\text{d}x,\\ \label{eq5.10}
        B_2(\tilde{u}_h,v)&=(\lambda+\mu)\int_{D} (\nabla\cdot \tilde{u}_h)(\nabla\cdot \bar{v})-(\nabla\tilde{u}_h:\mathcal{J}_{\mathcal{H}^{-1}_h})(\nabla\bar{{v}}:\mathcal{J}_{\mathcal{H}^{-1}_h}^\top)\det{\mathcal{J}_{\mathcal{H}_h}}\,\text{d}x,\\ \label{eq5.11}
        B_3(\tilde{u}_h,v)&=\omega^2\int_{D}\tilde{u}_h\cdot\bar{v}(\det{\mathcal{J}_{\mathcal{H}_h}}-1)\,\text{d}x.
    \end{align} Consider the Jacobi  matrix
     \begin{equation*}
        \mathcal{J}_{\mathcal{H}_h}=I_2+\nabla h.
    \end{equation*}
    Direct calculation implies \begin{equation}
       \label{eq5.12} \det{\mathcal{J}_{\mathcal{H}_h}}=1+
        \nabla \cdot h + O(\|h\|^2_{1,\infty}\|v\|_{H^1(D)^2})
    \end{equation}
    and \begin{equation} \label{eq5.13}
        \mathcal{J}_{\mathcal{H}^{-1}_h}=I_2-\nabla h+O(\|h\|^2_{1,\infty}\|v\|_{H^1(D)^2}).
    \end{equation}
    Combining \eqref{eq5.12} and \eqref{eq5.13} gives \begin{equation}\label{eq5.14}
\mathcal{J}^\top_{\mathcal{H}^{-1}_h}\mathcal{J}_{\mathcal{H}^{-1}_h}\det{\mathcal{J}_{\mathcal{H}_h}}=I_2-(\nabla h+\nabla h^\top)-(\nabla \cdot h )I_2 +O(\|h\|^2_{1,\infty}\|v\|_{H^1(D)^2}).
    \end{equation}
    Insert \eqref{eq5.12}-\eqref{eq5.14} into \eqref{eq5.9}-\eqref{eq5.11},
    \begin{align} \label{eq5.15}
       \sum_{i=1}^3 B_i(\tilde{u}_h,v) =  
        \sum_{i=1}^3 g_i(h)(\tilde{u}_h,v)+O(\|h\|^2_{1,\infty}\|v\|_{H^1(D)^2}),
    \end{align} where 
    \begin{align} \label{eq5.16}
        g_1(h)(\tilde{u}_h,v)&=\mu\int_D \sum_{j=1}^2\nabla \tilde{u}_{h,j}(\nabla h+\nabla h^\top -(\nabla \cdot h) I_2)\nabla \bar{v}_j\,\mathrm{d}x,\\ \label{eq5.17}
        g_2(h)(\tilde{u}_h,v)&=(\lambda+\mu)\int_D(\nabla \cdot \tilde{u}_h)(\nabla \bar{v} : \nabla h^\top)+(\nabla \cdot \bar{v})(\nabla \tilde{u}_h : \nabla h^\top)\notag\\&-(\nabla \cdot h)(\nabla \cdot \bar{v})(\nabla \cdot \tilde{u}_h)\,\mathrm{d}x, \\ \label{eq5.18}
        g_3(h)(\tilde{u}_h,v)&=\omega^2\int_D(\tilde{u}_h\cdot\bar{v})(\nabla\cdot h)\,\mathrm{d}x.
    \end{align}
    It is easy to verify \begin{equation} \label{eq5.19}
    |g_i(h)(\tilde{u}_h,v)-g_i(h)(u,v)|=O(\|h\|_{1,\infty}\|v\|_{H^1(D)^2}\|\tilde{u}_h-u\|_{H^1(D)^2}).
    \end{equation}
    Insert \eqref{eq5.19} into \eqref{eq5.15}-\eqref{eq5.18},
    \begin{align} \label{eq5.20}
     B(\tilde{u}_h-u,v)&=
     g_1(h)(u,v)+g_2(h)(u,v)+g_3(h)(u,v)+ \notag\\ &O(\|h\|^2_{1,\infty}\|v\|_{H^1(D)^2})+O(\|h\|_{1,\infty}\|v\|_{H^1(D)^2}\|\tilde{u}_h-u\|_{H^1(D)^2})
    \end{align}
    For $g_1(h)(u,v)$, applying the identity 
    \begin{align*}
    \nabla u(\nabla h+\nabla h^\top-(\nabla \cdot h)I_2)\nabla v =& \nabla \cdot \{h \cdot (\nabla u) \nabla v + (h\cdot \nabla v) \nabla u-(\nabla u\cdot \nabla v)h\} \\ &-(h\cdot \nabla v)\Delta u- (h \cdot \nabla u)\Delta v
    \end{align*} and divergence theorem gives \begin{align} \label{eq5.21}
        g_1(h)(u,v)&=-\mu \sum_{j=1}^2 \Big\{\int_D (h \cdot \nabla u_j)\Delta \bar{v}_j+(h \cdot \nabla \bar{v}_j)\Delta {u}_j \,\mathrm{d} x\notag \\
        &- \int_S(h\cdot \nabla u_j)(n \cdot \nabla \bar{v}_j)+(h\cdot \nabla \bar{v}_j)(n \cdot \nabla {u}_j)-(h\cdot n)(\nabla \bar{v}_j\cdot \nabla u_j) \, \mathrm{d}s\Big\}. 
    \end{align}
    Considering $(\Delta^*+\omega^2)u=0$, integration by parts gives \begin{align} \label{eq5.22}
     \sum_{j=1}^2 & \mu \int_D  (h \cdot \nabla \bar{v}_j)\Delta {u}_j \,\mathrm{d}x  = (\lambda+\mu)\int_D (\nabla \cdot u)\nabla \cdot (h \cdot \nabla \bar{v}) \,\mathrm{d}x \notag \\ &-(\lambda+\mu)\int_S(\nabla \cdot  u)(n \cdot (h\cdot \nabla \bar{v}))\,\mathrm{d}s-\omega^2 \int_D (h \cdot \nabla \bar{v}) \cdot u \,\mathrm{d}x.
    \end{align}
    Integration by parts again gives \begin{align} \label{eq5.23}
    \sum_{j=1}^2 \mu & \int_D (h \cdot \nabla u_j)\Delta \bar{v}_j \,\mathrm{d}x = \notag \\ &-\mu \int_{D} \nabla (h \cdot \nabla u) : \nabla \bar{v} \,\mathrm{d}x + \mu \int_S (h \cdot \nabla u) \cdot (n \nabla \bar{v}) \,\mathrm{d}s.
    \end{align}
    Denote $n=(n_1, n_2)^\top$ and $\tau=(-n_2, n_1)^\top $. Notice $v|_S=0$ implies $\partial_\tau v|_S=0$. This turns out that \begin{equation} \label{eq5.24}
        \sum_{i=1}^2 \int_S (h \cdot \nabla \bar{v})(n \cdot \nabla u_j)-(h \cdot n) (\nabla u_j \cdot \nabla \bar{v}_j) \,\mathrm{d}s =0.
    \end{equation}
    Applying the divergence theorem again and noticing $h|_\Gamma=0$, $v|_S=0$, we obtain \begin{align} \label{eq5.25}
        \int_D(\nabla \cdot h)(u\cdot \bar{v})+u \cdot (h \cdot \nabla \bar{v}) \,\mathrm{d}x &= \int_D \nabla \cdot((u \cdot \bar{v}) h)-(h \cdot \nabla u) \cdot \bar{v} \, \mathrm{d}x \notag \\ &=-\int_D(h \cdot \nabla u) \cdot \bar{v} \, \mathrm{d}x.
    \end{align}
    Inserting \eqref{eq5.21}-\eqref{eq5.25} into \eqref{eq5.16} and \eqref{eq5.18}  yields \begin{align} \label{eq5.26}
        g_1(h)(u,v) & +g_3(h)(u,v)= \notag \\ & \mu \int_D \nabla(h \cdot \nabla u):\nabla \bar{v}\,\mathrm{d}x-(\lambda+\mu)\int_D(\nabla \cdot u) \nabla \cdot (h \cdot \nabla \bar{v}) \mathrm{d} x \notag \\ & -\omega^2 \int_D (h \cdot \nabla u) \cdot \bar{v}\,\mathrm{d}x +(\lambda+\mu)\int_S (\nabla \cdot u)(n \cdot ( h \cdot \nabla \bar{v}) )\mathrm{d} s.
    \end{align}
    For $g_2(h)(u,v)$, direct calculation gives \begin{align} \label{eq5.27}
        \int_D ( \nabla \cdot u)( \nabla \bar{v} : \nabla h^\top) \,\mathrm{d}x=
        \int_D (\nabla \cdot u) \nabla  \cdot (h \cdot \nabla \bar{v})- (\nabla \cdot u)( h \cdot ( \nabla \cdot \nabla \bar{v}^\top))  \,\mathrm{d}x
    \end{align}
    and similarly \begin{align} \label{eq5.28}
        \int_D ( \nabla \cdot \bar{v})( \nabla u : \nabla h^\top) \,\mathrm{d}x=
        \int_D (\nabla \cdot \bar{v}) \nabla  \cdot (h \cdot \nabla u)- (\nabla \cdot \bar{v})( h \cdot ( \nabla \cdot \nabla u^\top))  \,\mathrm{d}x.
    \end{align}
    Then applying divergence theorem yields \begin{align} \label{eq5.29}
        &\int_D(\nabla \cdot \bar{v})( h \cdot ( \nabla \cdot \nabla u^\top))  \,\mathrm{d}x + \int_D (\nabla \cdot u)( h \cdot ( \nabla \cdot \nabla \bar{v}^\top))  \,\mathrm{d}x  \notag \\ &+\int_D (\nabla \cdot u)(\nabla \cdot \bar{v})(\nabla \cdot h) \,\mathrm{d}x = \int_D \nabla \cdot ( h ( \nabla \cdot u)(\nabla \cdot \bar{v})) \,\mathrm{d}x \notag \\ &= \int_S  ( h \cdot n )( \nabla \cdot u)(\nabla \cdot \bar{v}) \,\mathrm{d}s.
    \end{align}
    Recalling $\partial_\tau v|_S=0$, we have \begin{align} \label{eq5.30}
        \int_S  ( h \cdot n )( \nabla \cdot u)(\nabla \cdot \bar{v}) \,\mathrm{d}s= \int_S (\nabla \cdot u)( n \cdot ( h \cdot  \nabla \bar{v})) \mathrm{d}s.
    \end{align}
    Together with \eqref{eq5.27}-\eqref{eq5.30} implies \begin{align} \label{eq5.31}
        g_2(h)(u,v)&=(\lambda+\mu)\int_D (\nabla \cdot u) \nabla \cdot ( h \cdot \nabla \bar{v})+\nabla \cdot (h \cdot \nabla u)(\nabla \cdot \bar{v})\, \mathrm{d}x\notag \\&- (\lambda+\mu) \int_S( \nabla \cdot u)( n \cdot (h \cdot  \nabla \bar{v})) \,\mathrm{d}s.
    \end{align}
    Now recalling $w_h= h \cdot \nabla u$, $w_h|_\Gamma=0$ 
and combining the representations of $g_i(h)(u,v)$ in \eqref{eq5.26} and \eqref{eq5.31} implies 
\begin{align} \label{eq5.32}
        &\sum_{i=1}^3 g_i(h)(u,v)=\notag \\ &\mu \int_D \nabla (h \cdot \nabla u) : \nabla \bar{v} \,\mathrm{d}x +(\lambda+\mu)\int_D \nabla \cdot (h \cdot \nabla u) \cdot \bar{v}  \,\mathrm{d}x -\omega^2 \int_D (h \cdot \nabla u) \cdot \bar{v} \,\mathrm{d}x \notag \\ &= B(w_h,v).
    \end{align} 
    Notice $u^*$ satisfies \eqref{eq5.5}-\eqref{eq5.6}, so it satisfies the variation formula $B(u^*,v)=0$ for any $v \in V$. Hence by \eqref{eq5.20} and \eqref{eq5.32} we can obtain 
    \begin{equation} \label{eq5.33}
       \frac{B(\tilde{u}_h-u-u^*-w_h,v)}{\|h\|_{1,\infty}\|v\|_{H^1(D)^2}} = O(\|h\|_{1,\infty})+O(\|\tilde{u}_h-u\|_{H^1(D)^2}).
    \end{equation}
     The Dirichlet boundary condition \eqref{eq5.7} and $u, \tilde{u}_h \in V$  deduces $(\tilde{u}_h-u-u^*-w_h)|_S = 0$ so that $\tilde{u}_h-u-u^*-w_h \in V$.
     The using Theorem 3.1 we know the sesquilinear form $B$ generates an invertible bounded linear operator which is still denoted by $B$. By Banach open map theorem, $B^{-1}$ : $V^* \to V$ is also a bounded linear operator which implies \begin{equation} \label{eq5.34}
         \|\tilde{u}_h-u-u^*-w_h\|_{H^1(D)^2} \le C \sup_{v \in V, v\neq 0} \frac{|B(\tilde{u}_h-u-u^*-w_h,v)|}{\|v\|_{H^1(D)^2}}.
     \end{equation}
     Then combining \eqref{eq5.33}-\eqref{eq5.34} gives \[
     \frac{\|\tilde{u}_h-u-u^*-w_h\|_{H^1(D)^2}}{\|h\|_{1,\infty}}=O(\|h\|_{1,\infty})+O(\|\tilde{u}_h-u\|_{H^1(D)^2}).
     \] Hence we only need to prove when $\|h\|_{1,\infty} \to 0$, $\|\tilde{u}_h-u-u^*-w_h\|_{H^1(D)^2}\to 0$.
     Since $B^h(u,v)-B(u,v)=B_1(u,v)+B_2(u,v)+B_3(u,v)$. 
     Inserting \eqref{eq5.12}-\eqref{eq5.14} gives \[
     |B^h(u,v)-B(u,v)| \le C \|h\|_{1, \infty} \|u\|_{H^1(D)^2}\|v\|_{H^1(D)^2}.
     \] Since \[
     B^h(\tilde{u}_h,v)=(g,v)_\Gamma=B(u,v),
     \] we obtain\[
     \|\tilde{u}_h-u\|_{H^1(D)^2} \le C \sup_{u \in V, u \not= 0}\sup_{v \in V, v \not= 0} \frac{ |B^h(u,v)-B(u,v)|}{\|v\|_{H^1(D)^2}\|u\|_{H^1(D)^2}} \le C\|h\|_{1,\infty} \to 0.
     \]
     Hence  \[
     \lim_{\|h\|_{1,\infty}}  \frac{\|\tilde{u}_h-u-u^*-w_h\|_{H^1(D)^2}}{\|h\|_{1,\infty}}=0,
     \]which completes the proof.
\end{proof}

Next we consider a local stability result for inverse cavity problem. For two domains $D_1,\, D_2$ in $\mathbb{R}^2$, define the Hausdorff distance between $D_1$ and $D_2$ by \[
dist(D_1, D_2)=\max \{ \rho (D_1, D_2),\, \rho (D_2, D_1)\},
\] where \[\rho(D_1,D_2)=\sup_{x \in D_1}\inf_{y \in D_2} |x-y|.\] Take $h(x)=k p(x) n$, where $k \in \mathbb{R}$ and $p(x) \in C^2(S, \mathbb{R})$. It is easy to verify $dist (D_1,D_2)= O(h)$. Our goal is to give the following local stability result.
\begin{Th}
    Given $p \in C^2(S , \mathbb{R})$ and $k>0$ is sufficiently small, then\[
    dist (D_k,D) \le C\|u_k-u\|_{H^{1/2}(\Gamma)^2},
    \] where $D_k=D_{kpn}$, $u_k=u_{kpn}$ and $C>0$ is a constant independent with $k$.
\end{Th}
\begin{proof} For $p(x) \equiv 0$, it is obvious this conclusion is true. So we can assume $p \not\equiv 0$.

    Assume this conclusion is not true. There exists a $p(x) \in C^2(S , \mathbb{R}) $ such that \begin{equation} \label{eq5.35}
       \left\|\frac{u_{k'}-u}{k'}\right\|_{H^{1/2}(\Gamma)^2} \to 0, \quad \text{when} \quad k' \to 0,
    \end{equation} where $\{u_{k'}\}$ is some subsequence of $\{u_k\}$. Applying Theorem 5.1 to $h(x)= k p(x) n$ gives  \begin{equation} \label{eq5.36}
    \lim_{k \to 0} \left\|\frac{u_k-u-u^*}{k}\right\|_{H^{1/2}(\Gamma)^2}=0,
    \end{equation} where $u^*$ satisfies \begin{align*}
        (\Delta^*+\omega^2)u^*=0, \quad &{\rm in} \quad D, \\
        Tu^*=\mathcal{T} u^*, \quad &{\rm on}\quad \Gamma,  \\
        u^*=- hp \partial_n u, \quad &{\rm on} \quad S.
    \end{align*}  Let $v^*=u^*/h$, then $v^*$ satisfies \begin{align*}
        (\Delta^*+\omega^2)v^*=0, \quad &{\rm in} \quad D,  \\
        Tv^*=\mathcal{T} v^*, \quad &{\rm on}\quad \Gamma,  \\
        v^*=- p \partial_n u, \quad &{\rm on} \quad S.
    \end{align*}  Then by \eqref{eq5.36} we obtain \begin{equation} \label{eq5.37} \lim_{h \to 0} \left\|\frac{u_h-u}{h}-v^*\right\|_{H^{1/2}(\Gamma)^2}=0. \end{equation} Combining \eqref{eq5.35} and \eqref{eq5.37} yields $\|v^*\|_{H^{1/2}(\Gamma)^2}=0$. Hence \[
    Tv^*=\mathcal{T}v^*=0, \quad \text{on}\quad \Gamma.
    \] By unique continuation, \[v^*=0,\quad \text{in}\quad D.\]  Recall \[
    v^*=- p \partial_n u, \quad {\rm on} \quad S,
    \]  $p \not\equiv 0$ and $p \in C^2(S,\mathbb{R})$, there must exist 
 a continuous part $S' \subset S $ such that \[
    p \not= 0, \quad \text{on} \quad S'
    \] which implies \[
   \partial_n u = 0, \quad \text{on} \quad S'
    .\] Then by unique continuation $u=0$ in $D$ which is a contradiction to $Tu=\mathcal{T}u+g$ on $\Gamma$.
\end{proof}
\subsection{Neumann case}
Next consider the Neumann problem. Notice that the boundary value problem deduced by the NtD operator has the inhomogeneous boundary condition (2.20), so it is difficult to directly calculate the Fr$\acute{\rm e}$chet derivative like Dirichlet problem. Here, we still take the advantage of the TBC given by a DtN operator, though its definition on $\Gamma$ is difficult. Instead, a upper half-circle artificial boundary is taken, see in Figure \ref{fig:my_label2}. 
 \begin{figure}
    \centering
    \includegraphics{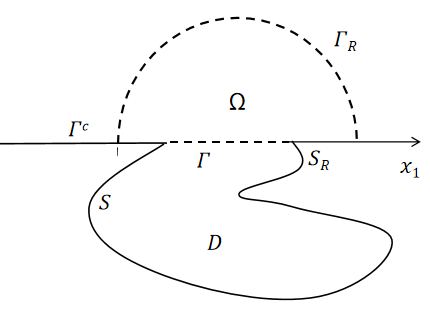}
    \caption{The Neumann case}
    \label{fig:my_label2}
\end{figure}

Take $R >0$ such that $\{x_2=0,|x|>R\}\subset \Gamma^c $ in  Figure \ref{fig:my_label2}. Denote $
 \Omega=D \cup \Gamma \{0<x_2,|x|<R\}
$, $\Gamma_R=\{x_2\ge 0,|x|=R\}$ and $S_R=S \cup (\Gamma^c \cap \{|x| \le R\})$. 
Let $x \in \mathbb{R}\cap \{ |x|>R\}$.

 Recalling the scattering wave $u_s$ of Neumann problem satisfies $Tu_s|_{\Gamma^c}=0$ and the Green's representation theorem yields \[
u_s= \int_{\Gamma_R} T_y G_N(x,y) \cdot u_s(y) - G_N(x,y) T_y u_s(y) \,\mathrm{d} s(y),
\] where $T_y$ is the differential operator $T$ with respect to $y$. Taking $x \to \Gamma_R$ gives \[
\frac{1}{2}u_s - \int_{\Gamma_R} T_y G_N(x,y) \cdot u_s(y) + G_N(x,y) \cdot  T_y u_s(y)=0.
\] Define the single-layer operator $\mathcal{S}$ by \[
\mathcal{S}v(x)=\int_{\Gamma} G_N(x,y) \cdot v(x) \,\mathrm{d}s(y), \quad  v \in H^{1/2}(\Gamma_R)
\] and double-layer operator $\mathcal{D}$ by\[
\mathcal{D}v(x)=\int_{\Gamma_R} T_yG_N(x,y) f(y) \,\mathrm{d}s(y), \quad  v \in H^{-1/2}(\Gamma_R).
\]
 We can choose $R$ such that $\omega^2$ 
 is not the interior Dirichlet eigenvalue of $\Delta^*$ so that $\mathcal{S}$ is invertible. Then define the DtN operator \[
 \mathscr{T} := -\mathcal{S}^{-1}(\frac{1}{2}\mathcal{I}-\mathcal{D}),
 \] where $\mathcal{I}$ is the identity operator. It is easy to verify this DtN operator implies the TBC \[
 Tu=\mathscr{T}u+g, \quad \text{on} \quad \Gamma_R
 \] with $g=T(u_i+u_r)-\mathscr(u_i+u_r)$.

  Then rewrite the boundary value problem as \begin{align*}
    (\Delta^*+\omega^2)u=0, \quad  &\text{in} \quad \Omega, \\ 
    Tu=0, \quad &\text{on} \quad  S_R, \\
    Tu=\mathscr{T} u+g, \quad &\text{on} \quad \Gamma_{c}.
\end{align*}
 The variation formula is given by \begin{equation} \label{eq5.38}
\mathscr{B}(u,v)=(g,v)_{\Gamma_R},\quad \forall v \in H^1(\Omega)^2,
\end{equation} where $B_{c}$ : $H^1(\Omega)^2 \times H^1(\Omega)^2 \to \mathbb{C}$ is defined by\[
\mathscr{B}(u,v)= \int_{\Omega} \mathcal{E}(u, \bar{v})-\omega^2 u \cdot \bar{v} \,\mathrm{d}x -\int_{\Gamma_R} \mathscr{T}u \cdot \bar{v} \,\mathrm{d}s.
\]
 The Dirichlet data on $\Gamma_R$ is directly given from the Neumann data on $\Gamma$ by \eqref{eq4.9}. Hence it is enough to consider the reconstruction from the Dirichlet data on $\Gamma_R$. 

 Let $\Omega_{h}= D_h \cup \Gamma \cup \{0<x_2, |x|<R\}$.
 Define the solution operator $\mathcal{U}_N$ : $C^2(S, \mathbb{R}^2) \to H^{1/2}(\Gamma_R)^2$ by \[
 \mathcal{U}_N (h)= u_h|_{\Gamma_R},
 \] where $u_h$ is solution to \begin{equation}\label{eq5.39}
 \mathscr{B}_h(u_h,v_h)=(g,v_h)_{\Gamma_R}, \quad \forall v_h \in H^1(\Omega_h)^2
 \end{equation} with \begin{equation} \label{eq5.40}
     \mathscr{B}_h(u_h,v_h)=\int_{\Omega_h} \mathcal{E}(u_h, \bar{v}_h)-\omega^2 u_h \cdot \bar{v}_h \,\mathrm{d}x -\int_{\Gamma_{R}} \mathcal{T}u_h \cdot \bar{v}_h \,\mathrm{d}s.
 \end{equation}
  Extend $h(x)$ to $C^2(\Omega,\mathbb{R}^2)$  and still denote it by $h$ such that  \[
h(x)=0, \quad \text{on} \quad \Gamma_{R},
\] and \begin{equation*} 
\|h\|_{H^{1,\infty}(\Omega)^2}\le C\|h\|_{H^{1,\infty}(S )^2}.
\end{equation*} with constant $C>0$. Define $\mathcal{H}_h$ by \[
\mathcal{H}_h(y)=y+h(y),\quad y \in \Omega.
\] It is invertible for sufficiently small $\|h\|_{H^{1,\infty}(\Omega)^2}$. 

Then taking the variable transform $x=\mathcal{H}_h(y)$ in \eqref{eq5.40} implies 
\begin{align*}
\mathscr{B}_h(u_h,v_h)=\mathscr{B}^h(\Tilde{u}_h, \Tilde{v}_h).\end{align*}
Here $\mathscr{B}^h$ is the similarly defined as $B^h$ in Section 5.1 except that $D_h$, $\Gamma$ and $\mathcal{T}$ are replaced by $\Omega_h$, $\Gamma_{R}$ and $\mathscr{T}$ respectively. Since $\Tilde{u}_h$, $\Tilde{v}_h \in H^1(\Omega)^2$, $\mathscr{B}^h$ is a sesquilinear form on $H^1(\Omega)^2 \times H^1(\Omega)^2$. So the variation formula \eqref{eq5.39} becomes
\begin{equation*} 
    \mathscr{B}^h( \Tilde{u}_h,v)=(g,v)_{\Gamma_R},\quad \forall v \in H^1(\Omega_h)^2.
\end{equation*} 
Let $\mathcal{U}_N'(0)$ : $C^2(S,\mathbb{R})^2 \to H^{1/2}(\Gamma_R)^2$ is the Fr$\acute{\rm e}$chet derivative of $\mathcal{U}_N$ at $0$. Then result similar as Theorem 5.1 holds for Neumann case.
\begin{Th}
   Suppose $h \in C^2(S, \mathbb{R})^2$ satisfying $h=0$ on $S \cap \Gamma$. Let $u^*$ is solution to the problem \begin{align} 
    \label{eq5.41}    (\Delta^*+\omega^2)u^*=0, \quad &{\rm in} \quad \Omega,  \\ 
       \label{eq5.42} Tu^*=\mathcal{T} u^*, \quad &{\rm on}\quad \Gamma_{R},  \\ 
       \label{eq5.43} Tu^*=\mathcal{A}, \quad &{\rm on} \quad S_R,
    \end{align}
    where $\mathcal{A}$ is a distribution defined on $H^{1/2}(S_R)^2$ by \begin{align*}
    \mathcal{A}(\bar{v})&= \omega^2 \int_{S_R} (u\cdot \bar{v})( h\cdot n) \,\mathrm{d}s \\ &-(\lambda+\mu)\int_{S_R}(\nabla \cdot u)(\nabla \cdot \bar{v})(n \cdot h)\,\mathrm{d}s - \mu \int_{S_R} (h\cdot n)(\nabla u :\nabla \bar{v}) \,\mathrm{d}s.
    \end{align*}
    Then $\mathcal{U}_N'(0)h=u^*|_{\Gamma_{c}}$.
\end{Th}
\begin{proof}
    Let $w_h=h \cdot \nabla u$. Then $h|_{\Gamma_R}=0$ implies $w_h|_{\Gamma_R}$. So $w_h+u^* \in H^1(\Omega)^2$ and $(w_h+u^*)|_{\Gamma_R}=u^*|_{\Gamma_R}$. Similarly as Theorem 5.1, we only need to prove\[
    \lim_{\|h\|_{1,\infty}\to 0} \frac{\|\tilde{u}_h-u-u^*-w_h\|_{H^1(\Omega)^2}}{\|h\|_{1,\infty}}=0,
    \]  where $\|\cdot\|_{1, \infty}=\|\cdot\|_{H^1(\Omega)^2}$.
    For any $v \in H^1(\Omega)^2$,\[
   \mathscr{B}(\tilde{u}_h-u-u^*-w_h,\phi)=\mathscr{B}(\tilde{u}_h-u,v)-\mathscr{B}(u^*+w_h,v).
    \] 
    Like Section 5.1, we have
    \begin{align} \label{eq5.44}
     \mathscr{B}(\tilde{u}_h-u,v)&=
     g_1(h)(u,v)+g_2(h)(u,v)+g_3(h)(u,v)+ \notag\\ &O(\|h\|^2_{1,\infty}\|v\|_{H^1(\Omega)^2})+O(\|h\|_{1,\infty}\|v\|_{H^1(\Omega)^2}\|\tilde{u}_h-u\|_{H^1(\Omega)^2}),
    \end{align} where $g_i(h)$ is defined the same as \eqref{eq5.16}-\eqref{eq5.18} except that $D$ is replaced by $\Omega$.
    
   For $g_1(h)(u,v)$, similarly as \eqref{eq5.21}, we obtain \begin{align} \label{eq5.45}
        g_1(h)(u,v)&=-\mu \sum_{j=1}^2 \{\int_{\Omega} (h \cdot \nabla u_j)\Delta \bar{v}_j+(h \cdot \nabla \bar{v}_j)\Delta {u}_j \,\mathrm{d} x\notag \\
        &- \int_{S_R}(h\cdot \nabla u_j)(n \cdot \nabla \bar{v}_j)+(h\cdot \nabla \bar{v}_j)(n \cdot \nabla {u}_j)\notag \\&-(h\cdot n)(\nabla \bar{v}_j\cdot \nabla u_j) \, \mathrm{d}s\}. 
    \end{align}
    By $(\Delta^*+\omega^2)u=0$ and integration by parts, we have \begin{align}\label{eq5.46} 
     \sum_{j=1}^2 & \mu \int_\Omega  (h \cdot \nabla \bar{v}_j)\Delta {u}_j \,\mathrm{d}x  = (\lambda+\mu)\int_\Omega (\nabla \cdot u)\nabla \cdot (h \cdot \nabla \bar{v}) \,\mathrm{d}x \notag \\ &-(\lambda+\mu)\int_{S_R}(\nabla \cdot  u)(n \cdot (h\cdot \nabla \bar{v}))\,\mathrm{d}s-\omega^2 \int_\Omega (h \cdot \nabla \bar{v}) \cdot u \,\mathrm{d}x.
    \end{align}
 Integration by parts again gives \begin{align}
 \label{eq5.47}
    \sum_{j=1}^2 \mu & \int_\Omega (h \cdot \nabla u_j)\Delta \bar{v}_j \,\mathrm{d}x = \notag \\ &-\mu \int_{\Omega} \nabla (h \cdot \nabla u) : \nabla \bar{v} \,\mathrm{d}x + \mu \int_{S_R} (h \cdot \nabla u) \cdot (n \cdot \nabla \bar{v}) \,\mathrm{d}s.
    \end{align}
Combining \eqref{eq5.45}-\eqref{eq5.47} and $(\lambda+\mu)(\nabla \cdot u)n+ \mu \partial_n u=0$ on $S_R$ gives
\begin{align} \label{eq5.48}
    g_1(h)(u,v)&=\mu \int_\Omega \nabla(h \cdot \nabla u):\nabla \bar{v}\,\mathrm{d}x-(\lambda+\mu)\int_\Omega(\nabla \cdot u) \nabla \cdot (h \cdot \nabla \bar{v}) \mathrm{d} x \notag  \\ &+\omega^2 \int_\Omega (h \cdot \nabla \bar{v} )\cdot u \,\mathrm{d}x- \mu \int_{S_R}(h\cdot n)(\nabla u : \nabla \bar{v}) \,\mathrm{d}s.
\end{align}
    Apply the divergence theorem again, we obtain \begin{align} \label{eq5.49}
        \int_\Omega(\nabla \cdot h)(u\cdot \bar{v})+u \cdot (h \cdot \nabla \bar{v}) \,\mathrm{d}x &= \int_\Omega \nabla \cdot((u \cdot \bar{v}) h)-(h \cdot \nabla u) \cdot \bar{v} \, \mathrm{d}x \notag \\ &=\int_{S_R} (u \cdot \bar{v})(h \cdot n)\,\mathrm{d}s-\int_\Omega(h \cdot \nabla u) \cdot \bar{v} \, \mathrm{d}x.
    \end{align}
    Combining \eqref{eq5.48}-\eqref{eq5.49} yields \begin{align} \label{eq5.50}
        g_1&(h)(u,v)  +g_3(h)(u,v)= \notag \\ & \mu \int_\Omega \nabla(h \cdot \nabla u):\nabla \bar{v}\,\mathrm{d}x-(\lambda+\mu)\int_\Omega(\nabla \cdot u) \nabla \cdot (h \cdot \nabla \bar{v}) \mathrm{d} x \notag \\ & -\omega^2 \int_\Omega (h \cdot \nabla u) \cdot \bar{v}\,\mathrm{d}x +\omega^2\int_{S_R} (u \cdot \bar{v})(h \cdot n)\,\mathrm{d}s -\mu \int_{S_R}(h\cdot n)(\nabla u : \nabla \bar{v}) \,\mathrm{d}s.
    \end{align}
   For $g_2(h)(u,v)$, similar discussion as \eqref{eq5.27}-\eqref{eq5.29} shows\begin{align} \label{eq5.51}
         g_2(h)(u,v)&=(\lambda+\mu)\int_\Omega (\nabla \cdot u) \nabla \cdot ( h \cdot \nabla \bar{v})+\nabla \cdot (h \cdot \nabla u)(\nabla \cdot \bar{v})\, \mathrm{d}x\notag \\&- (\lambda+\mu) \int_{S_R} (\nabla \cdot u)(\nabla \cdot \bar{v})( n\cdot h) \,\mathrm{d}s.
    \end{align}
    Notice $u^*$ satisfies \eqref{eq5.41}-\eqref{eq5.43}, then
\begin{align} \label{eq5.52}
B(u^*,v)&=\omega^2 \int_ {S_R}(u\cdot \bar{v})( h\cdot n) \,\mathrm{d}s \notag \\
&-(\lambda+\mu)\int_{S_R}(\nabla \cdot u)(\nabla \cdot \bar{v})(n \cdot h)\,\mathrm{d}s - \mu \int_{S_R} (h\cdot n)(\nabla u :\nabla \bar{v}) \,\mathrm{d}s.  
\end{align}
Combining $w_h|_{\Gamma_R}=0$ and \eqref{eq5.50}-\eqref{eq5.52} gives \begin{align*}
    &\sum^3_{i=1} g_i(h)(u,v)= 
    \mu \int_\Omega \nabla (h :\nabla u) \nabla \bar{v}\, \mathrm{d}x \\ &+(\lambda+\mu) \int_\Omega \nabla \cdot (\nabla u \cdot h)( \nabla \cdot \bar{v})\,\mathrm{d}x -(\lambda+\mu)\int_{S_R}(\nabla \cdot u)(\nabla \cdot \bar{v})(n \cdot h)\,\mathrm{d}s \\ &- \mu \int_{S_R} (h\cdot n)(\nabla u :\nabla \bar{v}) \,\mathrm{d}s + \omega^2 \int_{S_R} (u\cdot \bar{v})( h\cdot n) \,\mathrm{d}s=B(w_h+u^*,v).
\end{align*}
Then we arrive at \[
B(\tilde{u}_h-u-u^*-w_h,v)=O(\|h\|^2_{1,\infty}\|v\|_{H^1(\Omega)^2})+O(\|\tilde{u}_h-u\|_{H^1(\Omega)^2}\|v\|_{H^1(\Omega)^2}\|h\|_{1,\infty}).
\] to completes the proof.
\end{proof}
\emph{Remark}. We can not get local stability from the Fr$\acute{\rm e}$chet derivative like the Dirichlet problem. Let $h(x)=k p(x) n$ with $k>0$ and $p(x) \in C^2(S, \mathbb{R})$. Assume $p(x) \not\equiv 0$. For any $v \in H^1(\Omega)^2$, the equality \[
-(\lambda+\mu)\int_{S_R} p(\nabla \cdot u)(\nabla \cdot \bar{v})\,\mathrm{d}s - \mu \int_{S_R} p(\nabla u :\nabla \bar{v}) \,\mathrm{d}s + \omega^2 \int_ {S_R}p(u\cdot \bar{v}) \,\mathrm{d}s=0
\] does not imply $u|_{S_R}=0$. In fact, we can see in \cite{r23} for electromagnetic scattering, the local stability holds only for lossy medium, i.e, non-real wave number.

\section{Conclusion}
In this paper, elastic cavity problem with Dirichlet or Neumann condition is reduced into a bounded domain by the DtN or NtD operator. Variational approaches is utilized to prove the uniqueness and existence of boundary value problem in a bounded domain. For inverse cavity problem, the Fr$\acute{\rm e}$chet derivative is given for shape reconstruction and a local stability result is given for the Dirichlet problem. The stability results explicit with frequency for large elastic cavities like \cite{r21,r22} remains unsolved, which will be discussed in a forthcoming paper.

 \bibliographystyle{elsarticle-num} 
 \bibliography{cas-refs}





\end{document}